\documentclass[12pt,letterpaper]{article}

\usepackage{amsmath}
\usepackage{amssymb}
\usepackage{amsthm}
\usepackage{tikz-cd}
\usepackage{graphicx}
\usepackage{mathrsfs} 
\usepackage{mathtools} 
\usepackage{biblatex}
\newcommand{\Z}{\mathbb{Z}}
\newcommand{\Q}{\mathbb{Q}}
\newcommand{\R}{\mathbb{R}}

\newcommand{\GL}{\mathrm{GL}}
\newcommand{\SL}{\mathrm{SL}}
\newcommand{\Sp}{\mathrm{Sp}}
\newcommand{\Spf}{\Sp_4}
\newcommand{\bs}{\backslash}
\newcommand{\gk}{Gro{\ss}enkammer}
\newcommand{\gkn}{Gro{\ss}enkammern} 

\DeclareMathOperator{\vcd}{\mathrm{vcd}}

\newtheorem*{corollary}{Corollary}
\newtheorem{lemma}{Lemma}
\newtheorem{proposition}{Proposition}
\newtheorem{theorem}{Theorem}

\theoremstyle{definition}
\newtheorem{definition}{Definition}

\theoremstyle{remark}
\newtheorem*{remark}{Remark}

\begin{document}

\title{Computing Hecke Operators for Arithmetic Subgroups of $\Spf$}
\author{Dylan Galt \and Mark McConnell}

\maketitle


\section{Introduction} \label{sec:intro}

Let~$\mathcal{Q}$ be the space of positive definite real symmetric
bilinear forms in $n$~variables.  This is an open convex cone in the
vector space of real symmetric bilinear forms.  We
identify~$\mathcal{Q}$ with the positive definite $n\times n$
symmetric matrices.  Let $X_\SL$ be the quotient of~$\mathcal{Q}$ by
homotheties; this is the Riemannian symmetric space for $\SL_n(\R)$.
The group $\SL_n(\Z)$ acts properly discontinuously on $X_\SL$,
generalizing the classical action of $\SL_2(\Z)$ on the upper
half-plane.  Let~$\Gamma_\SL$ be an arithmetic subgroup of
$\SL_n(\Z)$.  Let~$\rho$ be a suitable local system of coefficients
on~$X_\SL$; the first lines of Section~\ref{sec:hos} will specify
which~$\rho$ we use.

The paper~\cite{MM20} introduced an algorithm for computing Hecke
operators on the equivariant cohomology $H^i_{\Gamma_\SL}(X_\SL;
\rho)$.  When~$\rho$ is over a field of characteristic zero, or of
characteristic not dividing the order of any torsion element
of~$\Gamma_\SL$, this is isomorphic to the ordinary cohomology
$H^i(\Gamma_\SL\bs X_\SL; \rho)$.  The algorithm in~\cite{MM20} works
for any~$\rho$ and for all $i = 0, 1, 2, \dots, \vcd(\Gamma_\SL)$,
where $\vcd(\Gamma_\SL)=\dim(\mathcal{Q}) - n = \frac12n(n-1)$ is the
virtual cohomological dimension.




The present paper extends~\cite{MM20} to the symplectic group for
$n=4$.  Let $\Spf(\R)$ be the subgroup of $\SL_4(\R)$ that preserves the
skew-symmetric bilinear form with matrix
\[
\Omega = \begin{bmatrix} 0&0&0&1 \\ 0&0&1&0 \\ 0&-1&0&0
  \\ -1&0&0&0 \end{bmatrix}.
\]
Let $X$ be the Riemannian symmetric space for $\Spf(\R)$. This is the
submanifold of~$\mathcal{Q}$ consisting of those $A\in\mathcal{Q}$
satisfying the symplectic condition $A \Omega A^t = \Omega$.  Working
mod homotheties, $X$ is embedded in~$X_\SL$.  Let~$\Gamma = \Gamma_\SL
\cap \Spf(\Z)$, where we always suppose $\Gamma_\SL$ is chosen so
that~$\Gamma$ is an arithmetic subgroup of $\Spf(\Z)$.  If~$\Gamma$ is
torsion free, $\Gamma\bs X$ is a smooth complex algebraic variety
called a Siegel modular threefold.

In this paper, we outline an algorithm for computing Hecke operators
on the equivariant cohomology $H^i_\Gamma(X; \rho)$.  The algorithm
works for any local coefficient system~$\rho$ and for all~$i$.

\subsection{Well-Tempered Complexes}

The algorithm for~$\SL_n$ in~\cite{MM20} uses the \emph{well-tempered
  complex} $\widetilde{W}^+$.  This is a regular cell complex of
dimension $\vcd(\Gamma_\SL) + 1$.  For a certain $\tau_0>1$, it is a
fibration $\widetilde{W}^+ \to [1, \tau_0^2]$, where the
coordinate~$\tau$ in the base is called the \emph{temperament}.  Let
$\widetilde{W}_\tau$ be the fiber over~$\tau$.  Each fiber is a
contractible cell complex of dimension $\vcd(\Gamma_\SL)$ on which
$\Gamma_\SL$ acts with finitely many orbits of cells.  The fiber
$\widetilde{W}_1$ is the well-rounded retract of \cite{Ash84}.
As~$\tau$ varies, there are a finite number of \emph{critical
  temperaments} $\tau^{(i)}$ where the cell structure of the fibers of
$\widetilde{W}^+$ abruptly changes.  On the intervals between
consecutive critical temperaments, the cell structure does not change
from fiber to fiber.  See Figures~\ref{fig:wrr2T2}
and~\ref{fig:wrr2T2near2} below for examples.

This paper's new algorithm for $\Spf$ uses a subcomplex
$\widetilde{V}^+$ of $\widetilde{W}^+$ for $n=4$.
This~$\widetilde{V}^+$ is a regular cell complex of dimension
$\vcd{\Gamma_\SL}+1$ and is a fibration $\widetilde{V}^+ \to [1,
  \tau_0^2]$.  Every fiber has dimension $\vcd{\Gamma_\SL} = 6$.  The
complex $\widetilde{V}^+$ and all its fibers admit an action
of~$\Gamma$ with only finitely many orbits of cells.  We define the
fiber $\widetilde{V}_1$ in Definition~\ref{defV1}; in the last
Section, we discuss how to compute the other fibers.

The $\widetilde{V}_\tau$ are not the complexes we would prefer to use.
\cite{MM93} introduced a cell complex $\widetilde{V}$ (called~$W$ in
that paper) whose dimension is~$4$, the true vcd of $\Spf(\Z)$.  The
complex~$\widetilde{V}$ is contractible and hence acyclic, and
$\Spf(\Z)$ acts on it with only finitely many orbits of cells.  In
\cite{MM89}, the combinatorics of the cells of~$\widetilde{V}$ are
described in terms of classical projective configurations in the
symplectic projective three-space $\mathbb{P}^3(\Q)$ endowed with the
form~$\Omega$.  Our~$\widetilde{V}_1$ in this paper is a
thickening\footnote{The notation was chosen because the letter~$V$ is
  thinner than~$W$.}  of~$\widetilde{V}$, of dimension~$6$.  More
precisely, it follows from \cite{MM93} that there is an
$\Spf(\Z)$-equivariant embedding of $\widetilde{V}$ as a subcomplex of
the first barycentric subdivision of $\widetilde{V}_1$.

Our main theorem is Theorem~\ref{thickening}, which says
that~$\widetilde{V}$ and $\widetilde{V}_1$ have the same cohomology.
This implies that $\widetilde{V}_1$ is itself an acyclic cell complex
on which $\Spf(\Z)$ acts with only finitely many stabilizers of cells.
As such, $\widetilde{V}_1$ is suitable for computing the equivariant
cohomology of~$\Gamma$.
The advantage of $\widetilde{V}_1$ over~$\widetilde{V}$ is that we can
extend $\widetilde{V}_1$ to $\widetilde{V}^+$, obtaining a Hecke
algorithm along the lines of \cite{MM20}.  The proof of
Theorem~\ref{thickening} appears in Section~\ref{sec:main}.

In Section~\ref{sec:wtcV}, we outline a computational method which,
conjecturally, would construct the fibers $\widetilde{V}_\tau$ for
$\tau > 1$ and show they are contractible.  Once these computations
were carried out, the rest of the Hecke operator algorithm would
proceed as in~\cite{MM20}.  We emphasize that Section~\ref{sec:wtcV}
is speculative, unlike the earlier sections.  Details for
Section~\ref{sec:wtcV} will appear in a later paper.

We summarize our notation.

\begin{center}
  \begin{tabular}{|c|l|}
    \hline
    $\widetilde{W}^+$ & well-tempered complex for $\SL_4(\mathbb{R})$ \\
    $\widetilde{W}_1$ & well-rounded retract for $\SL_4(\mathbb{R})$ at temperament~$1$ for $\widetilde{W}^+$ \\
    $\widetilde{V}$ & contractible complex for $\Spf(\mathbb{R})$ from \cite{MM93} \\
    $\widetilde{V}^+$ & the new acyclic subcomplex of $\widetilde{W}^+$ introduced in this paper \\
    $\widetilde{V}_1$ & cell complex at the first temperament for $\widetilde{V}^+$ \\
    \hline
    \end{tabular}
\end{center}


\subsection{Acknowledgments}

Avner Ash's paper~\cite{Ash84} is foundational for both~\cite{MM20} and
this paper.  Paul Gunnells suggested to us that combining~\cite{MM93}
and~\cite{MM20} might give a Hecke operator algorithm for $\Spf$.  We
thank both of them for these and many other helpful conversations.  We
also thank Robert MacPherson and Dan Yasaki.


\section{The Well-Tempered Complex for $\SL_n(\Z)$}
\label{sec:wtcz}

Here is a summary of~\cite{MM20}.  That paper concerns $\GL_n$ over
any division algebra~$D$ of finite dimension over~$\Q$.  We now
specialize to $D=\Q$, so that all arithmetic groups~$\Gamma$ are
subgroups of $\Gamma_0 = \GL_n(\Z)$.  Throughout this
Section~\ref{sec:wtcz}, we deal only with the objects called $X_\SL$
and $\Gamma_\SL$ in the Introduction, so we drop the subscripts~$\SL$
from those symbols.

A $\Z$-lattice in $\R^n$ is a finitely generated discrete subgroup that
contains an $\R$-basis.  $G = \GL_n(\R)$ acts on the right on row
vectors in $\R^n$, and $\Gamma_0 = \GL_n(\Z)$ stabilizes the standard
lattice $L_0 = \Z^n$.  Let $Y = \Gamma\bs G$.  We view~$Y$ as a space
of lattices, whose elements are $L_0 g$; the lattices have extra
structure, such as a level structure, when $\Gamma \subsetneqq
\Gamma_0$.  The group preserving the standard inner product $\langle
\,,\, \rangle$ on~$\R^n$ is the maximal compact subgroup $K =
\mathrm{O}_n \subset G$, and $X = G/K$ is the corresponding symmetric space.

\subsection{The Well-Rounded Retract}

\begin{definition} \label{arithmin}
  Let $L = L_0 g \in Y$.  The \emph{arithmetic minimum of}~$L$ is
  $m(L) = \min \{ \langle x, x\rangle \mid x \in L - \{0\}\}$.  The
  \emph{minimal vectors} are $M(L) = \{x \in L \mid \langle x,
  x\rangle = m(L) \}$.  We say~$L$ is \emph{well rounded} if $M(L)$
  spans $\R^n$.  The set of well-rounded lattices in~$Y$ with
  minimum~1 is denoted~$\widehat{W}$.
\end{definition}

The functions $m$ and~$M$ are $K$-invariant.  Hence $\widehat{W}$ is
$K$-invariant.

\begin{theorem}[{\cite[Thm.~2.11]{Ash84}}] \label{wrrmain}
  $W = \widehat{W}/K$ is a strong deformation retract of $Y/K$.  It is
  compact and of dimension $\vcd\Gamma_0$.  The universal
  cover\footnote{Strictly speaking, this is a ramified cover, because
    certain points of~$W$ have finite stabilizer subgroups
    in~$\Gamma_0$.  The barycentric subdivision in the last sentence
    of the theorem produces a triangulation that is compatible with
    the ramified covering map.}  $\widetilde{W}$ of~$W$ is a locally
  finite regular cell complex in~$X$ on which $\Gamma_0$ acts
  cell-wise with finite stabilizers of cells.  This cell structure has
  a natural barycentric subdivision which descends to a finite cell
  complex structure on~$W$.
\end{theorem}

\begin{definition}
  $W = \widehat{W}/K$ is the \emph{well-rounded retract}.
\end{definition}

\subsection{A Family of Retracts}
The paper~\cite{MM20} extends Theorem~\ref{wrrmain} by adding an extra
dimension to~$Y$.  It starts with the trivial bundle $Y \times I$ over
an interval~$I$, where $G$ acts fiberwise on $Y \times I$.  There is a
corresponding bundle isomorphism $(Y\times I)/K \cong (Y/K)\times I$
with fibers $Y/K$.

In order to generalize the construction of Theorem~\ref{wrrmain} and
build a family of retracts, one needs the concept of a family of
weights.  The quotient $\mathbb{P}^{n-1}(\Q)/\Gamma$ is finite.  A
\emph{set of weights} for~$\Gamma$ is a function\footnote{There is no
  implicit assumption of continuity for~$\varphi$; the only assumption
  on~$\varphi$ is $\Gamma$-invariance.} $\varphi :
\mathbb{P}^{n-1}(\Q)/\Gamma \to \R_+$. Such a~$\varphi$ defines a
\emph{set of weights} for~$L_0$, also denoted~$\varphi$, by
$\varphi(x) = \varphi(\Q x)$. This is a $\Gamma$-invariant function
$L_0 - \{0\} \to \R_+$.  For $L = L_0 g$, a set of weights~$\varphi$
for~$L_0$ defines a \emph{set of weights} for~$L$, by $\varphi^L(xg) =
\varphi(x)$, with $\varphi^L : L - \{0\} \to \R_+$.

A \emph{one-parameter family of weights} for~$L_0$ is a map
$\varphi_\tau : (L_0 - \{0\}) \times I \to \R^+$ which is a
$\Gamma$-invariant set of weights for any given~$\tau$, and for which
$\varphi_\tau(x)$ is real analytic in~$\tau$ for any given~$x$.  We
normalize $\varphi_\tau$ by dividing through by a positive real scalar, which depends continuously
on~$\tau$, so that the maximum of $\varphi_\tau$ is~1 for all~$\tau$. A one-parameter family of weights $\varphi_\tau$ determines
$\varphi^L_\tau$ for $L=L_0 g$ by $\varphi^L_\tau(xg) =
\varphi_\tau(x)$.  As a function of~$\tau$, the arithmetic minimum is given by
$m(L) = \min \{ \varphi_\tau^L(x) \langle x, x\rangle \mid x \in L -
\{0\}\}$, with minimal vectors
\begin{equation} \label{defMwithtau}
  M(L) = \{x \in L \mid \varphi_\tau^L(x)\langle x, x\rangle = m(L) \}.
\end{equation}

The spaces $\widehat{W}_\tau$ and $W_\tau = \widehat{W}_\tau/K$ for
any given~$\tau$ are defined similarly.  By \cite[Thm.~2.11]{Ash84},
there is a strong deformation retraction $R_\tau(t)$ of the fiber
over~$\tau$ onto $W_\tau$.  In fact, more is true:

\begin{theorem}[\cite{MM20}] \label{wtccont}
  $R_\tau(t)$ is a continuous map $((Y\times I)/K) \times [0,1]
  \to (Y\times I)/K$.
\end{theorem}

\begin{corollary}
  $\{(w \times \tau)/K \mid \tau \in I, w \in \widehat{W}_\tau\}$ is a
  strong deformation retract of $(Y \times I)/K$.  It has dimension
  $\vcd \Gamma$.  It is compact if~$I$ is compact.  The map from the
  retract to~$I$ is a fibration.
\end{corollary}

\subsection{Hecke Correspondences}

We review Hecke correspondences for~$\GL_n$, following \cite[\S3.1 and
  p.~76]{Sh}.
Define $\Delta = \{a \in G \mid L_0 a \subseteq L_0\}$.  Then
$\Gamma_0 \subset \Delta$, and~$\Delta$ is the sub-semigroup of
$\GL_n(\Q)$ with integer entries.  The arithmetic group $\Gamma =
\Gamma_0 \cap a^{-1} \Gamma_0 a$ is the common stabilizer in~$G$ of
$L_0$ and its sublattice $M_0 = L_0 a$.  One calls $(\Gamma_0,
\Delta)$ a \emph{Hecke pair}.

A point in $\Gamma_0\bs X$ has the form $\Gamma_0 g K$ with $g\in G$.
Define two maps
\begin{equation} \label{twomaps}
\begin{tikzcd}
    \Gamma\bs X \arrow[d, bend left, "q"] \arrow[d, bend right, "p"'] \\
    \Gamma_0\bs X
\end{tikzcd}
\end{equation}
by $p : \Gamma g K \mapsto \Gamma_0 g K$ and $q : \Gamma g K \mapsto
\Gamma_0 a g K$.  The \emph{Hecke correspondence} $T_a$ is the
one-to-many map $\Gamma_0\bs X \to \Gamma_0\bs X$ given by
\[
T_a = q \circ p^{-1}.
\]
It sends one point of $\Gamma_0\bs X$ to $[\Gamma_0 : \Gamma]$ points
of $\Gamma_0\bs X$, counting multiplicities.

The \emph{Hecke algebra} for the Hecke pair $(\Gamma_0, \Delta)$ is
the free abelian group on the set of correspondences~$T_a$ for
$a\in\Delta$, with multiplication defined by the composition of
correspondences.  This is equivalent to the traditional definition as
the algebra of double cosets $\Gamma_0 a\Gamma_0$ for $a\in\Delta$
\cite[p.~54]{Sh}.

For a prime $\ell\in\Z$ and for $k\in\{1,\dots,n\}$, define
\[
T_{\ell,k} = T_a \quad\mathrm{for\ } a =
\mathrm{diag}(\underbrace{1,\dots,1}_{n-k\mathrm{\ times}},
\underbrace{\ell,\dots,\ell}_{k\mathrm{\ times}}).
\]
The Hecke algebra is generated by the $T_{\ell,k}$ for all
primes~$\ell$ and $k\in\{1,\dots,n\}$.  If instead $G = \SL_n(\R)$ and
$\Gamma_0 = \SL_n(\Z)$, then $\Delta$ is the semigroup with entries
in~$\Z$ and positive determinant, and the Hecke algebra is generated
by the same $T_{\ell,k}$ \cite[\S3.2]{Sh}.

\begin{figure}
  \begin{center}
    \includegraphics[scale=0.3]{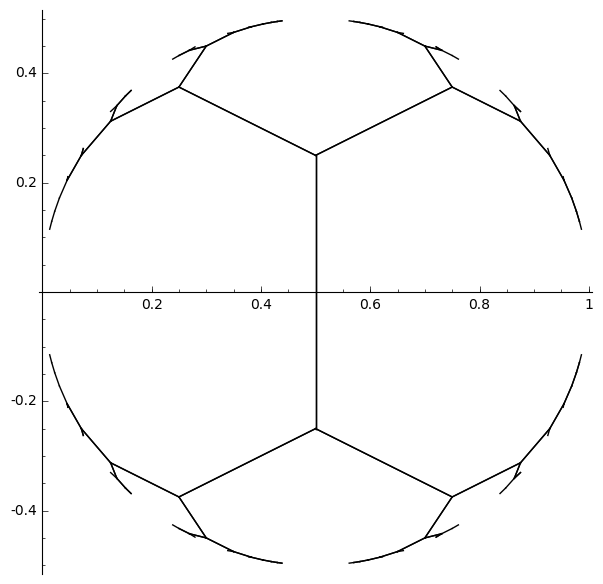} $\quad$
    \includegraphics[scale=0.3]{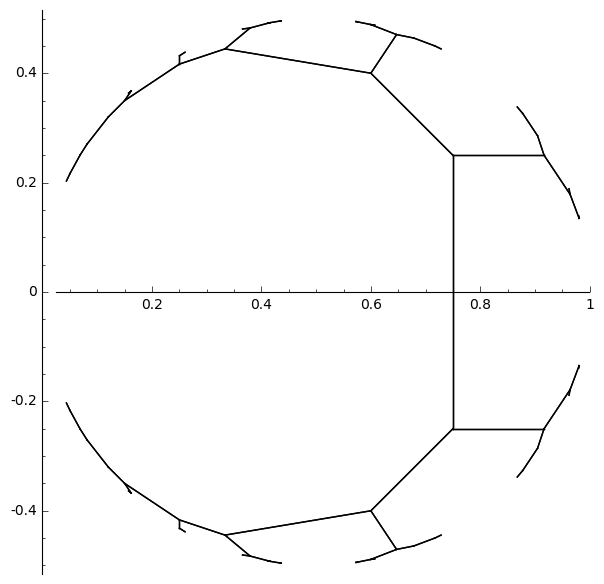}
  \caption{The well-rounded retract for $\GL_2(\Z)$, and its translate
    by $T_2$.}
  \label{fig:wrr2T2}
  \end{center}
\end{figure}

\subsubsection{Example for $n=2$} \label{sec:exWrr2T2}
In the Figures, we will present a running example for $\Gamma_0 =
\GL_2(\Z)$.  The left-hand side of Figure~\ref{fig:wrr2T2} shows the
complex~$\widetilde{W}$ for $\GL_2(\Z)$.  Here~$X$ is the unit disc,
which is the Klein model of the symmetric space.  $\widetilde{W}$ is a
tree.  $\Gamma_0$ acts on the tree, acting transitively on both the
vertices and the edges.

The right-hand side of Figure~\ref{fig:wrr2T2} shows the image
of~$\widetilde{W}$ under~$T_2 = T_{2,1}$.  It is a tree, and~$\Gamma$
is the largest subgroup of~$\Gamma_0$ that acts on it.  To
compute~$T_2$, we will build a one-parameter family of trees that
interpolates between the two sides of Figure~\ref{fig:wrr2T2} in a
$\Gamma$-equivariant way.  In the next section, we explain how to use
Theorem~\ref{wtccont} to build the family.
Figure~\ref{fig:wrr2T2near2} will show some members of the family.

\subsection{The Well-Tempered Complex}

Our choice of~$L_0$ determined the well-rounded retract
for~$\Gamma_0$.  Now fix $a\in\Delta$, and let $\Gamma = \Gamma_0 \cap
a^{-1} \Gamma_0 a$ as before.  The
well-tempered complex~$W^+$ will be determined by both $L_0$ and~$a$, and
will naturally admit an action by~$\Gamma$. 

Let $M_0 = L_0 a$.  By a standard calculation based on how~$M_0$
and~$\Gamma$ are defined in terms of~$a$, the next definition gives a
set of weights~$\varphi_\tau$ for~$\Gamma$.  We use this particular
set of weights for the rest of the paper.

\begin{definition} \label{deftausq}
  For $x \in L_0 - \{0\}$ and $\tau \geqslant 1$, define
\[
\varphi_\tau(x) = \left\{
\begin{array}{cl}
  \varphi(x) & \quad\mathrm{if\ }x \in M_0 - \{0\}, \\
  \tau^2\varphi(x) & \quad\mathrm{if\ }x \notin M_0.
\end{array} \right.
\]
\end{definition}

\begin{remark}
The idea here comes from $m(L)$ in Definition~\ref{arithmin}.  The
weighted squared length of a vector $x \in L$ is $\varphi^L(x) \langle
x, x \rangle$.  The squared length $\langle x, x \rangle$ scales by
$c^2$ when we multiply~$x$ by $c\in\R$.  By multiplying the weight
by~$\tau^2$ when $x \notin M_0$, we mimic the effect of scaling the
length of~$x$ linearly by~$\tau$.  We pretend $x\notin M_0$ gets
``longer by lies'', linearly.  When $x\in M_0$, we do not pretend to
change the length.
\end{remark}

Choose $\tau_0 > 1$, and let $I = [1, \tau_0]$.  The well-tempered
complex
depends on~$\tau_0$, but~\cite{MM20} shows that the complexes for two
different~$\tau_0$ are isomorphic when~$\tau_0$ is sufficiently large.

\begin{definition}
  The \emph{well-tempered complex}~$W^+$ for~$L_0$, $\varphi$, and~$a$
  is the image of $(Y \times [1,\tau_0])/K$ under the retraction
  $R_\tau(t)$ of Theorem~\ref{wtccont}, where~$\varphi_\tau$ is as in
  Definition~\ref{deftausq}.
\end{definition}

\begin{theorem}[{\cite[Thm.~4.33]{MM20}}] \label{wtccell}
  The universal cover~$\widetilde{W}^+$ of the well-tempered complex~$W^+$
  is a locally finite regular cell complex on which $\Gamma$ acts
  cell-wise with finite stabilizers of cells.  This cell structure has
  a natural barycentric subdivision which descends to a finite cell
  complex structure on~$W^+$.
\end{theorem}

\begin{figure}
  \begin{center}
    \includegraphics[scale=0.2]{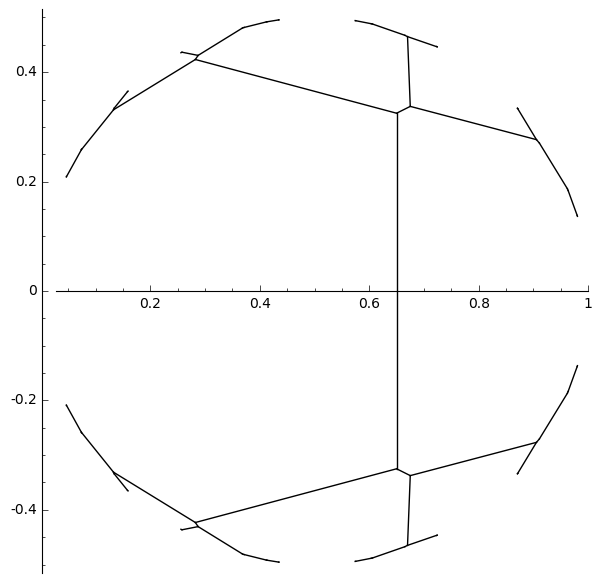}
    \includegraphics[scale=0.2]{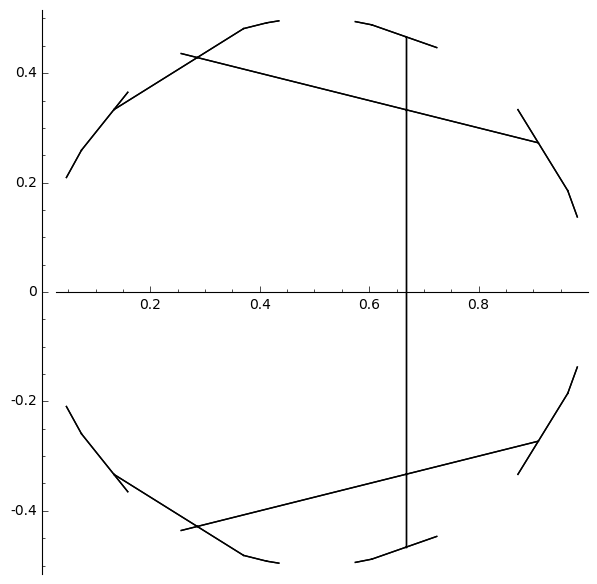}
    \includegraphics[scale=0.2]{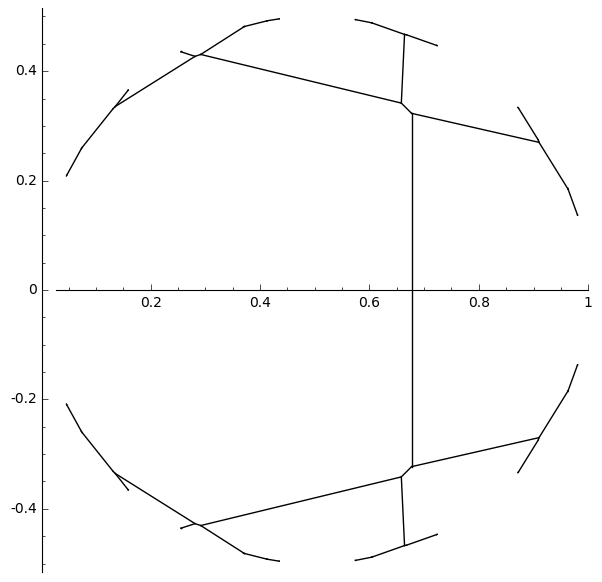}
  \end{center}
  \caption{How the fiber continuously deforms in the well-tempered complex.}
  \label{fig:wrr2T2near2}
\end{figure}

In the original well-rounded retract $\widetilde{W}$, the cells are
indexed by their sets of minimal vectors~$M$, each of which is a finite subset of $L_0-\{0\}$.  In the well-tempered complex, cells are
indexed by pairs consisting of sets~$M$ and a set of temperaments.
The proof of Theorem~\ref{wtccell} in~\cite{MM20} shows that there are
a finite number of \emph{critical temperaments} $\tau^{(i)}$ with $1 =
\tau^{(0)} < \tau^{(1)} < \cdots < \tau^{(r)} = \tau_0$.  The
cells~$\sigma$ of Theorem~\ref{wtccell} are cut into closed pieces
along the hyperplanes $\tau = \tau^{(i)}$ for $i = 0, \dots, r$.  Each
non-empty cell of the refinement is indexed by a pair.  The pair is
$(M, [\tau^{(i-1)}, \tau^{(i)}])$ if the projection of the cell to the
$\tau$-coordinate is $[\tau^{(i-1)}, \tau^{(i)}]$.  The pair is $(M,
[\tau^{(i)}, \tau^{(i)}])$ if the projection is $\{\tau^{(i)}\}$.  We
will write $[\tau, \tau']$ as shorthand for both $[\tau^{(i-1)},
  \tau^{(i)}]$ and $[\tau^{(i)}, \tau^{(i)}]$.

\subsubsection{Example for $n=2$}

We continue the example from Section~\ref{sec:exWrr2T2} for~$T_2$ for
$\Gamma_0 = \GL_2(\Z)$.  The critical temperaments turn out to be
$\tau^{(i)} = 1, 2, 4$.  The well-tempered complex $\widetilde{W}^+$
has dimension~2.  Figure~\ref{fig:wrr2T2} showed the slices of
$\widetilde{W}^+$ at~$\tau = 1$ and~$4$.  Figure~\ref{fig:wrr2T2near2}
shows the slices at $\tau = 2-\varepsilon$, $2$, and $2+\varepsilon$
for a small $\varepsilon>0$.  It illustrates how the cell structure
changes at $\tau=2$.


\subsubsection{Hecketopes}
Voronoi’s reduction theory~\cite{Vor} gives a way to make the
well-rounded retract~$\widetilde{W}$.  The Voronoi cones of~\cite{Vor}
are the cones over the faces of a Voronoi polyhedron.  The cells
of~$\widetilde{W}$ are unions of cells in a certain subdivision of the
Voronoi cones, and, in fact, the cells of~$\widetilde{W}$ are dual to
the faces of the Voronoi polyhedron.  In the same way, the
well-tempered cells of $\widetilde{W}^+$ are dual to a generalization
of the Voronoi polyhedron called the \emph{Hecketope}.  Section~6
of~\cite{MM20} describes the Hecketope in full, presenting practical
algorithms for finding the cells of~$\widetilde{W}^+$ along with the
critical temperaments and the indexing data $(M, [\tau, \tau'])$.

\subsubsection{The first and last temperament}
For the~$a$ giving the Hecke operator $T_{\ell,k}$,~\cite{MM20} sets
$\tau_0 = \ell$ and shows there is then a simple relationship between
the fibers of the well-tempered complex over~$\tau_0$ and over~$1$:

\begin{theorem}[\cite{MM20}] \label{wtctau0}
  For any $\tau \geqslant \tau_0$, the map $X \to X$ given by $gK
  \mapsto a^{-1} gK$ descends mod~$\Gamma$ to give a cell-preserving
  homeomorphism from the well-rounded retract $W_1$ over~$1$ to the
  well-rounded retract $W_\tau$ over~$\tau$.  If a cell over $\tau =
  1$ is $\sigma_1(Q)$ with index set $Q \subset L_0 - \{0\}$, then the
  cell that corresponds to $\sigma_1(Q)$ under the homeomorphism has
  index set $Qa$.
\end{theorem}

We call the endpoints of $[1, \tau_0]$ the \emph{first} and
\emph{last} temperaments, respectively.

\subsection{Computing Hecke Operators}
\label{sec:hos}

Let the Hecke pair $(\Gamma_0, \Delta)$ be as above.  Let~$\rho$ be
any left $\Z\Delta$-module.  (We often take the tensor product
of~$\rho$ with a field like $\Q$ or $\mathbb{F}_p$.)  There is a
natural left action of the Hecke algebra on the equivariant cohomology
$H^*_{\Gamma_0}(X; \rho)$ \cite[\S1.1]{AS}.  For $a\in\Delta$, the
action of the Hecke correspondence~$T_a$ on the cohomology is called
the \emph{Hecke operator} associated to $a$, and it will also be
denoted~$T_a$.  It is defined to be $p_* q^*$ in a diagram
derived from~\eqref{twomaps}:
\begin{equation} \label{twomapsequivar}
\begin{tikzcd}
    H^*_{\Gamma}(X; \rho) \arrow[d, bend right, "p_\ast"'] \\
    H^*_{\Gamma_0}(X; \rho) \arrow[u, bend right, "q^\ast"'] 
\end{tikzcd}
\end{equation}
The map $q^* : H^*_{\Gamma_0}(X; \rho) \to H^*_{\Gamma}(X; \rho)$ is
the natural pullback map for~$q$. The map $p_* : H^*_{\Gamma}(X; \rho)
\to H^*_{\Gamma_0}(X; \rho)$ is the transfer map \cite[III.9]{Br}
for~$p$, which is defined because $\Gamma = \Gamma_0 \cap a^{-1}
\Gamma_0 a$ has finite index in~$\Gamma_0$.

We now give an algorithm that uses the well-tempered complex to
compute~$T_a$. To compute equivariant cohomology, we may use any
acyclic cell complex on which~$\Gamma_0$ acts.  The fiber
$\widetilde{W}_\tau$ of the well-tempered complex $\widetilde{W}^+$
over any~$\tau$ is a strong deformation retract of~$X$, hence acyclic.
This holds in particular for the fibers $\widetilde{W}_{\tau^{(i)}}$
over the critical temperaments $\tau^{(i)}$, and for the inverse image
of the closed interval between two consecutive critical
temperaments. Indeed, $\widetilde{W}_{[\tau^{(i-1)}, \tau^{(i)}]}$ has
dimension one higher than the vcd, but its cohomology in degree
vcd$+1$ will be zero.

First, we compute~$p_*$.  We use $\tau=1$, the first temperament, when
working with~$p$.  The retracts $\widetilde{W}$ and $\widetilde{W}_1$
are equal by definition. $\Gamma_0$ acts on $\widetilde{W}$, and the
smaller group $\Gamma$ acts on $\widetilde{W}_1$.  Computing the
transfer map is straightforward.  (In practice it is tricky to get the
orientation questions correct.  This is true for all the cells, and
especially for the cells with non-trivial stabilizer subgroups.  This
comment applies to all the computations in this paper.)

Next, we compute~$q^*$.  The pullback map is natural on cohomology,
but we must account for the factor of~$a$ in the definition of~$q$.
The key is to use the \emph{last} temperament~$\tau_0$ when working
with~$q$.  We compute $H^*_\Gamma(X; \rho)$ as
$H^*_\Gamma(\widetilde{W}_{\tau_0}; \rho)$.  By Theorem~\ref{wtctau0},
there is a homeomorphism of cell complexes $\widetilde{W}_{\tau_0} \to
\widetilde{W}_1$, from the last temperament to the first, given by
multiplication by~$a$.  As we saw for~$p$, $\widetilde{W}_1$
equals~$\widetilde{W}$.  Thus there is a cellular map which enables us
to compute $q^* : H^*_{\Gamma_0}(\widetilde{W}; \rho) \to
H^*_\Gamma(\widetilde{W}_{\tau_0}; \rho)$.

Computing only $p_*$ and $q^*$ does not give us the Hecke operator.
The map of Theorem~\ref{wtctau0} involves dividing or multiplying
by~$a$.  It is not a map of $\Gamma$-modules, because $a\in\Delta$ but $a \notin
\Gamma$ in general.  For this reason, we cannot directly map
$H^*_\Gamma(\widetilde{W}_{\tau_0}; \rho)$ to
$H^*_\Gamma(\widetilde{W}_1; \rho)$.  To overcome this last
difficulty, we use the whole well-tempered complex to define a chain
of morphisms and quasi-isomorphisms.  For $i=1,\dots, i_r$, in the
portion $\widetilde{W}_{[\tau^{(i-1)},\tau^{(i)}]}$ over the fibers $\tau \in
[\tau^{(i-1)},\tau^{(i)}]$, define the closed inclusions of the fibers on
the left and right sides:
\[
  \widetilde{W}_{\tau^{(i-1)}}
  \xhookrightarrow{l^{(i-1)}}
  \widetilde{W}_{[\tau^{(i-1)},\tau^{(i)}]}
  \xhookleftarrow{r^{(i)}}
  \widetilde{W}_{\tau^{(i)}}
  \]
By Theorems~\ref{wtccont} and~\ref{wtccell}, we can compute the
pullbacks $(l^{(i-1)})^*$ and the pushforwards $(r^{(i)})_*$ on
$H^*_\Gamma(\dots; \rho)$.  The pullback is a naturally defined
cellular map.  The pushforward $(r^{(i)})_*$ is a quasi-isomorphism,
the inverse of the pullback $(r^{(i)})^*$; we compute the pullback at
the cochain level using the cellular map, then invert the map on
cohomology.

We summarize our algorithm as a theorem.

\begin{theorem}[\cite{MM20}] \label{wtcalgor}
  With notation as above, the Hecke operator $T_a$ on equivariant
  cohomology~\eqref{twomapsequivar} may be computed in finite terms as
  the composition
  \begin{equation} \label{plrlrq}
  p_* {l^{(0)}}^* r^{(1)}_* {l^{(1)}}^* r^{(2)}_* \cdots {l^{(i_r-1)}}^* r^{(i_r)}_* q^* .
  \end{equation}
\end{theorem}


\subsection{Cohomology of Subgroups}
\label{subsec:levelN}

Let $\Gamma' \subseteq \Gamma_0$ be an arithmetic subgroup.  We wish
to compute Hecke operators on the equivariant cohomology
$H^*_{\Gamma'}(X; \rho)$ for any~$\Gamma'$.  By Shapiro's Lemma
\cite[III.6.2]{Br}, $H^*_{\Gamma'}(X; \rho) \cong H^*_{\Gamma_0}(X;
\mathrm{Coind}_{\Gamma'}^{\Gamma_0}\rho)$.  We use
Theorem~\ref{wtcalgor} to compute the latter.


\section{A Subcomplex for $\Spf(\Z)$}
\label{sec:main}

\subsection{PL Embedding Lemma}
The well-rounded retract~$\widetilde{W}$ for $\SL_4(\R)$ has real
dimension~6.  All of its 6-cells are equivalent modulo $\SL_4(\Z)$; as
a representative 6-cell, we may choose the cell~$\sigma$ whose minimal
vectors are the columns of the $4\times 4$ identity matrix~\cite{Vor}.

\begin{definition} \label{defV1}
Denote by $\widetilde{V}_1$ the following closed subcomplex of $\widetilde{W}$:
\[
\widetilde{V}_1 = \left \{ \gamma\cdot\sigma \; | \; \gamma\in\Spf(\Z)
\right \}.
\]
\end{definition}
$\widetilde{V}_1$ has
an action of $\Spf(\Z)$, but not an action of $\SL_4(\Z)$. We will denote by~$\alpha$ a closed cell in $\widetilde{V}_1$ that is
a non-empty intersection of the form
\begin{equation} \label{capManyV1Cells}
    \alpha = \alpha_{i_1\cdots i_k}=\bigcap^k_{j=1}\gamma_{i_j} \sigma,
    \quad \gamma_{i_j}\in \Spf(\Z).
\end{equation}
We will use this notation to suppress indices wherever they do not
play a crucial role.

Let $\widetilde{V}$ be the retract for $\Spf(\R)$ constructed
in~\cite{MM93}. The following lemma allows us to
identify~$\widetilde{V}$ with its embedded image inside the subcomplex
$\widetilde{V}_1$ of $\widetilde{W}$ that we have just defined.

\begin{lemma} \label{lemmaPLemb}
There exists a PL-embedding $\widetilde{V}\rightarrow\widetilde{V}_1$.
\end{lemma}
\begin{proof}
  In both $\widetilde{W}$ and $\widetilde{V}$, the cells are in
  one-to-one correspondence with their sets of minimal vectors.  In
  either cell complex, a cell~$\alpha$ is a face of~$\beta$ if and
  only if the set of minimal vectors for~$\alpha$ contains the set of
  minimal vectors for~$\beta$, by~\cite{M91} and~\cite{MM93}.  Denote
  by $\mathscr{C}_{\textup{SL}}$ the poset of the sets of minimal
  vectors for~$\widetilde{W}$ and by $\mathscr{C}_{\textup{Sp}}$ the
  corresponding poset for $\widetilde{V}$.  These are ranked posets,
  where the rank of an item is the dimension of the corresponding
  cell.  By the construction in~\cite{MM93}, there is an injective
  homomorphism of ranked posets $\mathscr{C}_{\textup{Sp}}\rightarrow
  \mathscr{C}_{\textup{SL}}$. Since geometric realization is a
  faithful functor, it follows that there is PL-embedding
  $\widetilde{V}\rightarrow \widetilde{W}$, whose image is contained
  in $\widetilde{V}_1$.
\end{proof}

\subsection{Thickening Theorem}

We remarked in the introduction that the PL embedding
$\widetilde{V}\rightarrow\widetilde{V}_1$ is a thickening
of~$\widetilde{V}$, raising the dimension from~4 to~6.  The main
theorem of this Section is that the two spaces have the same topology.

\begin{theorem} \label{thickening}
  The PL embedding $\widetilde{V}\rightarrow\widetilde{V}_1$ induces
  an isomorphism on cohomology.  In particular, $\widetilde{V}_1$ is acyclic.
\end{theorem}

\subsection{Local Contractibility}

We need a local result about contractibility.  In the next section,
this will be extended to prove the global result that
$\widetilde{V}_1$ is acyclic.

\begin{proposition} \label{adhesions}
  For any~$\alpha$ of the form~\eqref{capManyV1Cells}, $\alpha\cap
  \widetilde{V}$ is a contractible subcomplex of~$\widetilde{V}_1$.
\end{proposition}


\begin{proof}
Without loss of generality, we may assume $\alpha$ is a face of
$\sigma$. Indeed, by its definition, $\widetilde{V}_1$ is invariant
under $\Spf(\Z)$, so we may replace~$\alpha$ by $\gamma\alpha$ for any
$\gamma\in\Spf(\Z)$.  After this replacement, we may take
$\gamma_{i_1} = I$.

  Let $\mathscr{R}$ be the set of all cells $\widetilde{W}$ which have
  the form $\gamma\sigma$ for some $\gamma\in\Spf(\Z)$ and such that
  $\gamma\sigma \cap \sigma \ne\varnothing$.  By definition,
  $\mathscr{R}$ is a subset of $\widetilde{V}_1$.  It is finite, by
  the local finiteness of $\widetilde{W}_1$.  Every non-empty~$\alpha$
  of the form~\eqref{capManyV1Cells} will have all of its
  $\gamma_{i_j}\sigma$ in~$\mathscr{R}$, given the constraint
  $\alpha\subseteq\sigma$.

  We use a computer to enumerate and store~$\mathscr{R}$, as follows.
  Enumerate all the faces~$\beta$ of~$\sigma$ (these have dimensions
  $0,\dots,6$).  For each~$\beta$, let~$M$ be the set of its minimal
  vectors; $M$ is a subset of $\Z^4$ containing between~4 and~12
  vectors.  (We find~$M$ based on the tables in~\cite{M91}.  Vectors
  $\vec{x}$ and $-\vec{x}$ in~$M$ are counted only once.)  For
  each~$M$, consider all $\binom{|M|}{4}$ four-element subsets $M_4$.
  We test whether we can permute the columns of $M_4$, and multiply
  zero or more of its columns by~$-1$, to make $M_4\in\Spf(\Z)$.  If
  the test passes, then $\gamma = M_4 \in\Spf(\Z)$ is such that
  $\gamma\sigma\in\mathscr{R}$.

  Next, we compute all $\alpha$'s by computing all $k$-fold
  intersections of cells in~$\mathscr{R}$.  We use a hash table whose
  value is an~$\alpha$ as in~\eqref{capManyV1Cells}, and whose key~$M$
  is the union of the minimal vectors for the $\gamma_{i_j}$ appearing
  in the intersection.  (In other words, $M$ is the union of the
  column vectors~$\vec{x}$ in the matrices $\gamma_{i_1} = I$,
  $\gamma_{i_2}$, \dots, $\gamma_{i_k}$, and $-\vec{x}$ too.)  We use
  a loop to fill the hash table first with ($k=1$)-fold intersections
  (which means $\gamma_{i_1} = I$ only), then ($k=2$)-fold
  intersections, then $k=3$, etc.  When a value~$\alpha$ becomes the
  empty cell, we stop exploring that branch of the table.

  Consider one of the~$\alpha$ in the table.  As we have said,
  $\alpha$ is a PL cell, hence is contractible.  What the proposition
  asserts is that $\alpha\cap \widetilde{V}$ is contractible.  Let~$B$
  be the set of sets of minimal vectors $M_\beta$ for all
  faces~$\beta$ of~$\sigma$ which contain~$\alpha$ and such
  that~$M_\beta$ is one of the sets of minimal vectors occurring
  in~$\widetilde{V}$.  In terms of Lemma~\ref{lemmaPLemb}, each $M_\beta \in B$
  determines a vertex in the image of the PL embedding, and the
  containment relations among the sets determine a simplicial
  subcomplex $\alpha_\triangle$ of the image of the PL embedding.
  This subcomplex $\alpha_\triangle$ is $\alpha\cap \widetilde{V}$.

  Showing, for each~$\alpha$, that $\alpha_\triangle$ is contractible
  is a matter of direct checking.  The first possibility is that the
  minimal vectors of~$\alpha$ already determine a cell
  in~$\widetilde{V}$; then $\alpha_\triangle$ is homeomorphic to the
  first barycentric subdivision of~$\alpha$ itself, hence is
  contractible.  The second possibility is that $\alpha_\triangle$ is
  a single closed simplex; obviously this is contractible.  The third
  possibility is that $\alpha_\triangle$ is a more general finite
  simplicial complex.  Here we use computation to verify three facts
  about $\alpha_\triangle$: its reduced homology with coefficients
  in~$\Z$ is trivial, its fundamental group is trivial, and it is
  shellable.  For a finite simplicial complex, trivial $\Z$-homology
  together with trivial fundamental group imply $\alpha_\triangle$ has
  the homotopy type of a point; this gives one proof that
  $\alpha_\triangle$ is contractible.  A second proof is that a
  shellable complex is a bouquet of spheres, and trivial $\Z$-homology
  implies the number of spheres in the bouquet is zero.
\end{proof}

\subsubsection{Performance of the algorithm}
The computation in the previous proof was coded up in
Sage~\cite{SageMath}.  In the last paragraph of the proof, when
$\alpha_\triangle$ was a ``more general finite simplicial complex'',
we checked that its reduced $\Z$-homology was trivial, that its
fundamental group~$\pi_1$ was trivial, \emph{and} that it was
shellable.  As the proof says, checking shellability was unnecessary
given the first two.  Nevertheless, we were curious to see how the
$\pi_1$ and shellability algorithms would perform, so we used them
both.

The code completed, proving Proposition~\ref{adhesions}, in seven
days.  Without checking $\pi_1$ and shellability, it would have
completed in less than~24 hours.  The largest sets~$M$ encountered had
$|M|=8$.

\subsection{Proof of the Thickening Theorem}

We recall results about second derived neighborhoods. Let $K$ be
a simplicial complex. For a simplex $A\in K$, the \textit{star of A in
  K} is the following open subcomplex of $K$:
\[
    \textup{star}(A;K)=\left \{ B\in K \; | \; B \geq A \right \}
\]
where the relation $\geq$ is cellular inclusion. Its closure  $\overline{ \textup{star}}(A;K)$ comprises the cells of $\textup{star}(A;K)$ and their faces. 

A subcomplex $K_0\subseteq K$ is called \textit{full} if no simplex of $K-K_0$ has all of its vertices in $K_0$. The \textit{closed simplicial neighborhood} of a full subcomplex $K_0$ in $K$ is formed by taking the following union of closed stars:
\[
    N(K_0;K)=\bigcup_{\textup{vertices} \; v\in K_0}\overline{\textup{star}}(v;K)
\]
\indent
Denote by $\left |  N(K_0;K) \right |$ the underlying polyhedron of this closed simplicial neighborhood. If $K_0\subseteq K$ is a full subcomplex, then $\left |  N(K_0;K) \right |$ is referred to as a \textit{derived neighborhood} of the polyhedron $\left |  K_0 \right |$ in the PL-manifold $\left |  K \right |$. More generally, let $K^{(r)}$ be the $r^\textup{th}$ barycentric subdivision of the complex $K$. Then, for a full subcomplex $K_0\subseteq K$, the polyhedron $\left |  N(K^{(r)}_0;K^{(r)}) \right |$ is the $r^\textup{th}$ \textit{derived neighborhood} of $\left |  K_0 \right |$ in $\left |  K \right |$. That is:
\[
    N(K^{(r)}_0;K^{(r)})=\bigcup_{\textup{vertices} \; v\in K^{(r)}_0}\overline{\textup{star}}(v;K^{(r)})
\]

\begin{theorem} [{\cite[Thm.~2.11]{Hu}}] \label{dylanHudson}
The second derived neighborhood of a full subcomplex $K_0\subseteq K$ is a regular neighborhood of $\left |  K_0 \right |$ in the PL-manifold $\left |  K \right |$. In particular, it is a strong deformation retract of $\left |  K_0 \right |$.
\end{theorem}

With these preliminaries, we return to the proof of the main theorem.
With respect to a fixed triangulation of the well-rounded
retract~$\widetilde{W}$, the complex $\widetilde{V}_1$ is a simplicial
subcomplex of the first barycentric subdivision $\widetilde{W}^{(1)}$.
By~\cite{Ash84} and~\cite{MM20}, each closed cell $\alpha \in
\widetilde{V}_1$ is convex.  By convexity, any simplex of
$\widetilde{W}^{(1)}$ having all of its vertices in~$\alpha$ must be
contained in~$\alpha$, since a simplex is the convex hull of its
vertices.  Therefore, $\alpha$ is a full simplicial subcomplex of
$\widetilde{V}_1$, and Theorem~\ref{dylanHudson} applies.

Let $\widetilde{V}_1^{(2)}$ denote the second barycentric subdivision
of $\widetilde{V}_1$. For each closed cell $\alpha\in
\widetilde{V}_1$, form the simplicial subcomplexes
$N(\alpha^{(2)};\widetilde{V}_1^{(2)})$ of $\widetilde{V}_1^{(2)}$,
and denote by $N_\alpha$ the corresponding second derived
neighborhood. By Theorem~\ref{dylanHudson}, $N_\alpha$ is a regular
neighborhood of $\alpha$ in $\widetilde{V}_1$, whence its interior
$N^\circ_\alpha$ is a strong deformation retract of
$\alpha$. Moreover, we have the following lemma:
\begin{lemma} \label{lemmaNhbd}
For distinct cells $\alpha_1,\alpha_2\in \widetilde{V}_1$ with common face $\alpha_1\cap\alpha_2=\alpha$ one has:
\[
    N^\circ_{\alpha_1}\cap N^\circ_{\alpha_1}=N^\circ_\alpha
\]
\end{lemma}
\begin{proof}
The result follows directly from the observation that
\[
     N(\alpha^{(2)}_1;\widetilde{V}_1^{(2)})\cap  N(\alpha^{(2)}_2;\widetilde{V}_1^{(2)})=N(\alpha^{(2)};\widetilde{V}_1^{(2)}).
\]
Indeed, recall that
\[
    N(\alpha^{(2)};\widetilde{V}_1^{(2)})=\bigcup_{\textup{vertices} \; v\in \alpha^{(2)}}\overline{\textup{star}}(v;\widetilde{V}_1^{(2)}).
\]
Since $\alpha$ is the common face of $\alpha_1$ and $\alpha_2$, the vertices of $\alpha^{(2)}$ are precisely the common vertices of $\alpha^{(2)}_1$ and $\alpha^{(2)}_2$, justifying the desired equality.
\end{proof}

By Lemma~\ref{lemmaNhbd}, the union of the $N^\circ_\alpha$ for each
closed cell $\alpha\in \widetilde{V}_1$ is a \v{C}ech cover of
$\widetilde{V}_1$. Thus, by a generalized Mayer-Vietoris argument in
relative homology \cite[p.~161]{BT}, we obtain a proof of the main
theorem, as follows.

\begin{proof}[Proof of Theorem~\ref{thickening}]
By Proposition~\ref{adhesions}, $H_n(N^\circ_\alpha,N^\circ_\alpha\cap \widetilde{V})=0$ for all degrees $n$. Then, from the long exact sequence of the pair $(N^\circ_\alpha,N^\circ_\alpha\cap \widetilde{V})$ in relative homology there is an isomorphism $ H_n(N^\circ_\alpha)\cong H_n(N^\circ_\alpha,N^\circ_\alpha\cap \widetilde{V})$ in all degrees, whence $H_n(N^\circ_\alpha,N^\circ_\alpha\cap \widetilde{V})=0$ for all $n$. Now, consider the relative homology of the pair $(\widetilde{V}_1,\widetilde{V})$, where $\widetilde{V}$ is identified with its image under the piecewise linear embedding constructed in Lemma~\ref{lemmaPLemb}. We claim that $H_n(\widetilde{V}_1,\widetilde{V})=0$ for all degrees $n$. Let $\mathfrak{U}$ denote the \v{C}ech open cover of $\widetilde{V}_1$ consisting of the $N^\circ_\alpha$. Denote by $N^\circ_{\alpha_{i_0\cdots i_k}}$ the open polyhedral neighborhood corresponding to the intersection $\alpha_{i_0\cdots i_k}$, which is well-defined by Lemma~\ref{lemmaNhbd}. The augmented double complex $C_\ast((\mathfrak{U},\widetilde{V}),A_\ast)$ endowed with the differential $D=\delta+(-1)^p\cdot d$ computes the singular relative homology $H_\ast(\widetilde{V}_1,\widetilde{V})$. This double complex has groups:
\[
    K_{p,q}=\prod_{i_0<\cdots<i_p} A_q(N^\circ_{\alpha_{i_0\cdots i_p}},N^\circ_{\alpha_{i_0\cdots i_p}}\cap \widetilde{V})
\]
with $A_q$ the $q^{\textup{th}}$ singular relative homology group. By Proposition~\ref{adhesions}, the vertical $d$-complexes are exact, and by the generalized Mayer-Vietoris principle, so are the horizontal $\delta$-complexes. Therefore, the spectral sequence of this double complex degenerates at the $E^2$ page, and we have $H_n(\widetilde{V}_1,\widetilde{V})=0$ in all degrees $n$. Finally, the long exact sequence of the pair $(\widetilde{V}_1,\widetilde{V})$ in relative homology gives an isomorphism $H_n(\widetilde{V})\cong H_n(\widetilde{V}_1)$ in all degrees. But, by~\cite{MM93} we know $\widetilde{V}$ is contractible, whence $\widetilde{V}_1$ is acyclic.
\end{proof}


\section{A Well-Tempered Complex for $\Spf$}
\label{sec:wtcV}

In the previous section, we defined a closed subcomplex
$\widetilde{V}_1$ of $\widetilde{W}_1$.  Our~$\widetilde{V}_1$ is
acyclic, and (by definition) it has an action of $\Spf(\Z)$ with only
finitely many orbits of cells.  In Section~\ref{subsec:wtcVdef}, we
describe how one could extend this to all temperaments, defining a
closed subcomplex $\widetilde{V}^+$ of $\widetilde{W}^+$, so that
$\Spf(\Z)$ acts on $\widetilde{V}^+$ with only finitely many orbits of
cells, and so that for each temperament~$\tau$ the fiber
$\widetilde{V}_\tau$ of $\widetilde{V}^+$ over~$\tau$ is acyclic.  The
definition of $\widetilde{V}^+$ would proceed by induction on~$i$ from
one critical temperament~$\tau^{(i)}$ to the next.
Section~\ref{sec:end} outlines a Hecke operator algorithm based on
this for arithmetic subgroups of $\Spf(\Z)$.

We emphasize that Section~\ref{sec:wtcV} is speculative, unlike
Sections~\ref{sec:intro}--\ref{sec:main}.  Details will appear in a
later paper.

\subsection{Defining the Well-Tempered Complex for $\Spf$}
\label{subsec:wtcVdef}

Extending the definition up to a critical temperament is relatively
straightforward.  At a critical temperament $\tau^{(i)}$ for $i > 0$,
we define the cells of $\widetilde{V}_{\tau^{(i)}}$ to be the cells of
$\widetilde{W}_{\tau^{(i)}}$ that are in the closure of those for
$\widetilde{V}_{[\tau^{(i-1)}, \tau^{(i)}]}$.

To start the induction at $i=0$, we note that the first temperament
$\tau^{(0)} = 1$ is not technically a critical temperament.
When~$\tau$ is~$\geqslant 1$ but very near~$1$,
Formula~\eqref{defMwithtau} shows that the sets of minimal
vectors~$M(L)$ do not change.  They will not change until~$\tau$
reaches some specific value, which is $\tau^{(1)} > 1$.  The cells of
$\widetilde{V}^+$ over $[\tau^{(0)}, \tau^{(1)}]$ are in one-to-one
correspondence with those over $\tau^{(0)} = 1$, locally cylindrical
extensions of one higher dimension.  The passage to $\tau^{(1)}$ can
thus be handled as in the previous paragraph.

When we extend by closure from the cells over $\tau \in (\tau^{(i-1)},
\tau^{(i)})$ to the closure over $\tau^{(i)}$, our inductive
hypothesis is that $\widetilde{V}_{[\tau^{(i-1)}, \tau^{(i)}]}$ is an
acyclic complex.  We need to prove that $\widetilde{V}_{\tau^{(i)}}$
is also acyclic.  It suffices to work modulo a torsion-free arithmetic
subgroup of $\Spf(\Z)$, such as $\Gamma(3)$.  By looking at the sets
of minimal vectors, we will define a cellular map $\Gamma(3)\bs
\widetilde{V}_{[\tau^{(i-1)}, \tau^{(i)}]} \to \Gamma(3)\bs
\widetilde{V}_{\tau^{(i)}}$.  We anticipate that this cellular map
will be a cellular collapsing map, but we will need to prove it is a
collapsing map.  One way to do this is by discrete Morse theory
\cite{For98} \cite{For02}.  The quotients $\Gamma(3)\bs
\widetilde{V}_{[\tau^{(i-1)}, \tau^{(i)}]}$ and $\Gamma(3)\bs
\widetilde{V}_{\tau^{(i)}}$ are finite, and they are regular cell
complexes.  We will put a discrete Morse function on $\Gamma(3)\bs
\widetilde{V}_{\tau^{(i)}}$.  We anticipate being able to extend it in
some sensible way to a function on $\Gamma(3)\bs
\widetilde{V}_{[\tau^{(i-1)}, \tau^{(i)}]}$, for instance by adding
new Morse values for $\Gamma(3)\bs \widetilde{V}_{[\tau^{(i-1)},
    \tau^{(i)}]}$ in the same order that they appear in $\Gamma(3)\bs
\widetilde{V}_{\tau^{(i)}}$.  Once the function has been extended, it
is straightforward to see whether the extension is a discrete Morse
function that defines a collapsing map.  If it is not, we will study
the failure and improve the extended function on an ad hoc basis.

Extending the definition from $\widetilde{V}_{\tau^{(i)}}$ to
$\widetilde{V}_{[\tau^{(i)}, \tau^{(i+1)}]}$, for $i>0$, requires more
care.  There are many cells in $\widetilde{W}_{[\tau^{(i)},
    \tau^{(i+1)}]}$ whose closures meet $\widetilde{V}_{\tau^{(i)}}$,
but we only want to take some of them,
the smallest possible set so that $\widetilde{V}_{[\tau^{(i)},
    \tau^{(i+1)}]}$ will be acyclic and of dimension~$7$.  Certainly
we will include all top-dimensional cells~$\mathscr{T}$ of
$\widetilde{W}_{[\tau^{(i)}, \tau^{(i+1)}]}$ whose closures meet
$\widetilde{V}_{\tau^{(i)}}$ in a top-dimensional cell in codimension
one; here the sets of minimal vectors are not changing as $\tau$
increases across the codimension-one boundary (another locally
cylindrical case).  Examples show, however, that there can be holes
in~$\mathscr{T}$; the complex may not be acyclic.

We will make a provisional definition of $\widetilde{V}_{[\tau^{(i)},
    \tau^{(i+1)}]}$, and then will fill the holes in~$\mathscr{T}$, if
there are any, by the following procedure.  Let~$P$ be the Borel
subgroup of upper-triangular matrices in $\Spf(\R)$, and $P(\Z)$ its
integer points.  $P(\Z)\bs P(\R)$ is a nilmanifold whose universal
cover is $P(\R)$, homeomorphic to~$\R^4$.  Let $\sigma_4$ be the
top-dimensional cell in~$\widetilde{V}$ whose minimal vectors are the
columns of the identity matrix; every 4-cell in~$\widetilde{V}$ is
equivalent to it.  Define the \emph{standard \gk}\footnote{The name
  means \emph{great chamber} in the Tits building.  More accurately,
  it is a particular gallery in that building, determined by the
  minimal parabolic~$P$.} in $\widetilde{V}$ to be $\{\gamma\sigma_4
\mid \gamma\in P(\Z)\}$.  This is homeomorphic to the universal cover
of the nilmanifold $P(\Z)\bs P(\R)$.

Define the \emph{standard \gk} in $\widetilde{V}_1$ to be
$\{\gamma\sigma \mid \gamma\in P(\Z)\}$.  Intuitively, this is a
thickening of the standard \gk\ in $\widetilde{V}$.  It is
homeomorphic to $\R^4\times\R^2$, with an action of $P(\Z)$ on the
$\R^4$ factor, and the quotient modulo $P(\Z)$ is a trivial
$\R^2$-bundle over the nilmanifold.

In either $\widetilde{V}$ or $\widetilde{V}_1$, a \emph{\gk}
is~$\gamma$ times the standard \gk, for any $\gamma\in\Spf(\Z)$.  By
Definition~\ref{defV1}, $\widetilde{V}_1$ is the union of the
\gkn\ coming from all translates by coset representatives of
$\Spf(\Z)/ P(\Z)$.

The Borel subgroup $P(\R)$ is a maximal solvable subgroup.  It is
filtered by a sequence of normal subgroups so that the subquotients
are copies of the additive group~$\R$.  One such sequence is
\begin{equation} \label{nilsubgps}
\cdots \subset
\left[\begin{smallmatrix} 1&0&0&* \\ 0&1&*&0 \\ 0&0&1&0
    \\ 0&0&0&1 \end{smallmatrix}\right]
\subset
\left[\begin{smallmatrix} 1&0&*&* \\ 0&1&*&* \\ 0&0&1&0
    \\ 0&0&0&1 \end{smallmatrix}\right]
\subset
P(\R).
\end{equation}
These subgroups foliate~$P(\R)$ by copies of $\R^3$, which are in turn
foliated by copies of $\R^2$, which are in turn foliated by copies of
$\R^1$.

To fill the holes in~$\mathcal{T}$, we will first find appropriate
definitions of the standard \gk\ for $\widetilde{V}_{[\tau^{(i)},
    \tau^{(i+1)}]}$.  We will consider the action of $P(\Z)$ on sample
top-dimensional cells in~$\mathcal{T}$, choosing them so that they
fill out as much of the thickened~$\R^4$ as possible.  It will be
easiest to act on cells in~$\mathcal{T}$ by the one-dimensional
subgroup in~\eqref{nilsubgps}, making cellular models of the leaves
$\R^1$.  If the model is a thickened $\R^1$ with gaps, it will be easy
to see which cells fill in those gaps.  Next, we will act by the the
two-dimensional subgroup in~\eqref{nilsubgps}, making cellular models
of the leaves $\R^2$, and so on.  We will perform these checks for
temperaments $i=0, 1, 2, \dots$.

At the end, we will have a provisional definition of a \gk, a cellular
model of a thickened~$\R^4$.  We will define
$\widetilde{V}_{[\tau^{(i)}, \tau^{(i+1)}]}$ to be the union of these
provisional \gkn\ coming from all translates by coset representatives
of $\Spf(\Z)/ P(\Z)$.  Since the definition is provisional, it will
be necessary to prove that $\widetilde{V}_{[\tau^{(i)},
    \tau^{(i+1)}]}$ is acyclic.  We can do this using discrete Morse
theory as described above.

\subsection{Outline of a Hecke Operator Algorithm for $\Spf$}
\label{sec:end}

By \cite[Thms. ~3.37 and~3.40]{Andr}, the Hecke algebra for $\Spf(\Z)$
is generated by the Hecke correspondences $T_a$ where we take the
following two~$a$'s for each prime~$\ell$:
\[
\mathrm{diag}(1, 1, \ell, \ell),\quad
\mathrm{diag}(1, \ell, \ell, \ell^2).
\]
(There is a change of coordinates, because~\cite{Andr} uses the
symplectic form $\left[\begin{smallmatrix} 0 & I \\ -I &
    0 \end{smallmatrix}\right]$ rather than our~$\Omega$.)  The
subgroups $\Gamma = \Spf(\Z) \cap a^{-1} \Spf(\Z) a$, when reduced
mod~$\ell$, are the Siegel and Klingen parabolics, respectively.
By~\cite[Thm.~4]{MM20}, $\tau_0 = \ell$ and $\ell^2$ in the respective
cases.

To compute the Hecke operators, we replace $\widetilde{W}^+$ with
$\widetilde{V}^+$ and compute Formula~\eqref{plrlrq} in this setting.


\printbibliography

\end{document}